\documentclass{amsart}
\usepackage{epsfig}
\usepackage[english]{babel}
\usepackage{amssymb}
\usepackage{amsfonts}
\usepackage{enumerate}
\usepackage{amsmath}
\usepackage{graphicx}
\usepackage[babel=true]{csquotes}

\usepackage{etoolbox} \patchcmd{\section}{\scshape}{\bfseries}{}{}
\makeatletter
\renewcommand{\@secnumfont}{\bfseries}
\makeatother

\newcommand{\R}{\mathbb{R}}
\newcommand{\E}{\mathbb{E}}

\newcommand{\1}{\mathbb{1}}
\renewcommand{\P}{\mathbb{P}}
\newcommand{\Z}{\mathbb{Z}}
\newcommand{\N}{\mathbb{N}}
\newcommand{\T}{\mathbb{T}}
\newcommand{\C}{\mathbb{C}}

\newcommand{\A}{\mathcal{A}}

\newcommand{\M}{\mathcal{M}}
\newcommand{\pri}{\mathcal{P}}

\newcommand{\Q}{\mathcal{Q}}

\newcommand{\B}{\mathcal{B}}

\newcommand{\F}{\mathcal{F}}

\newcommand{\setdef}{\stackrel {\rm {def}}{=}}
\newcommand{\ds}{\displaystyle}

\def\1{\,\rlap{\mbox{\small\rm 1}}\kern.15em 1}

\def\build#1_#2^#3{\mathrel{\mathop{\kern 0pt#1}\limits_{#2}^{#3}}}
\def\tend#1#2{\build\hbox to 12mm{\rightarrowfill}_{#1\rightarrow #2}^{ }}

\def\tendn{\tend{n}{\infty}}
\def\converge#1#2#3#4{\build\hbox to
#1mm{\rightarrowfill}_{#2\rightarrow #3}^{\hbox{\scriptsize #4}}}

\def\Rea#1{{\rm {Re}}({#1})}

\def\rand{\rm {rand}}

\newcommand{\beq}{\begin{equation}}
\newcommand{\eeq}{\end{equation}}

\numberwithin{equation}{section}
\theoremstyle{plain}
\newtheorem{thm}{Theorem}[section]

\newtheorem{lemm}[thm]{Lemma}

\newtheorem{defn}[thm]{Definition}

\newtheorem{conj}[thm]{Conjecture}

\newtheorem{Prop}[thm]{Proposition}

\newtheorem{rem}[thm]{Remark}
\newtheorem{Cor}[thm]{Corollary}
\newtheorem*{sarnak}{Sarnak's conjecture}

\begin{document}

\title[Spectral properties of the M\"{o}bius function]{Spectral properties of the M\"{o}bius function and a random M\"{o}bius model.}

\author{E. H. el Abdalaoui }

\author {M. Disertori }

\address{Department of Mathematics, University\\
of Rouen, LMRS, UMR 60 85, Avenue de l'Universit\'e, BP.12, 76801\\
Saint Etienne du Rouvray - France}
\email{elhoucein.elabdalaoui@univ-rouen.fr }
\email{margherita.disertori@univ-rouen.fr}

\maketitle

{\renewcommand\abstractname{Abstract}}

\begin{abstract}
 Assuming Sarnak conjecture is true for any singular dynamical process, we prove that the spectral measure of the M\"{o}bius function is equivalent to Lebesgue measure. Conversely, under Elliott conjecture, we establish that the M\"{o}bius function is orthogonal to any uniquely ergodic dynamical system with singular spectrum. Furthermore, using Mirsky Theorem, we find a new simple proof of Cellarosi-Sinai Theorem on the orthogonality of the square of the M\"{o}bius function with respect to any
weakly mixing dynamical system. Finally, we establish Sarnak conjecture for a
particular random model.
\vspace{8cm}

\hspace{-0.7cm}{\em AMS Subject Classifications} (2000): 37A15, 37A25, 37A30.
\end{abstract}

\newpage

\section{Introduction}
\paragraph{}

The M\"{o}bius function is defined for the positive integers $n$ by
\begin{equation}\label{Mobius}
\mu(n)= \begin{cases}
 1 {\rm {~if~}} n=1; \\
(-1)^r  {\rm {~if~}} n
{\rm {~is~the~product~of~}} r {\rm {~distinct~primes}}; \\
0  {\rm {~if~not}}
\end{cases}
\end{equation}
It is of great importance in Number Theory because of its  connection with the
Riemann $\zeta$-function via the formulae
\[
\sum_{n=1}^{+\infty}\frac{\mu(n)}{n^s}=\frac{1}{\zeta(s)}, \qquad
\sum_{n=1}^{+\infty}\frac{|\mu(n)|}{n^s}=\frac{\zeta(s)}{\zeta(2s)}
~~{\rm {with}}~~~ \Rea{s}>1.
\]
Furthermore, the estimate
\[
\left|\ds \sum_{n=1}^{x}\mu(n)\right|=O\left(x^{\frac12+\varepsilon}\right)\qquad
{\rm as} \quad  x \longrightarrow +\infty,\quad \forall \varepsilon >0
\]
is equivalent to the Riemann Hypothesis (\cite[pp.315]{Titchmarsh}).\medskip

The main goal of this note is to investigate the problem of spectral disjointness
of the M\"{o}bius function with any dynamical sequence with zero topological entropy.
The question is  initiated by P. Sarnak in \cite{sarnak1}
\cite{sarnak2}.
In his three lectures notes  \cite{sarnak}, P. Sarnak makes the following conjecture.
\begin{sarnak}\label{conj-sarnak}
 The M\"obius function is orthogonal to any deterministic sequence $(a_n)_{n \in \N}$,
that is,
\begin{eqnarray}\label{sarnak-conj}
\frac{1}{N}\sum_{n=1}^{N} \mu(n)a_n \tendn 0.
\end{eqnarray}
\end{sarnak}
The sequence $(a_n)$ is deterministic if it is generated by
a deterministic topological dynamical system $(X,T)$, i.e. $X$ is a compact space,
$T$ a continuous map from $X$ onto $X$ with topological entropy equal
to zero and there exists a continuous function from X to the complex plan $\C$ and
a point $x \in X$ for which  $a_n=f(T^nx)$ for all $n$.
\medskip

The M\"{o}bius function is orthogonal to any constant fonction \cite{Hildebrand}. This is follows
 from Kronecker's Lemma combined with Landau's observation:
\[
\sum_{n \geq 1}\frac{\mu (n)}{n} =0.
\]
This last relation also implies that the orthogonality of the M\"{o}bius function to the function $1$
(i.e. $\frac{1}{x}\sum_{n \leq x}\mu(n)
\tend{x}{+\infty}0$) is \enquote{equivalent} to the Prime Number Theorem  (PNT for short),
which says  that the number $\pi(x)$ of primes below $x$ satisfies
\[
\lim_{x \longrightarrow \infty}  \frac{\pi(x)\log(x)}{x } =1.
\]

The orthogonality of the M\"obius function to any sequence arising from a rotation
dynamical system ($X$ is the circle $\T$ and $Tx=x+\alpha$, $\alpha \in \T$) follows
 from the following inequality (Davenport \cite{Da})
\[
\max_{\theta \in \T}\left|\displaystyle\sum_{k \leq x}\mu(k)e^{ik\theta}\right|
\leq \frac{x}{\log(x)^{\varepsilon}},
\qquad {\rm where}\  \varepsilon >0.
\]
It is an easy exercise to establish, from Davenport's estimate above combined with the
spectral theorem and the standard ergodic argument \cite{Dunford-Schawrtz},
that the M\"obius function is orthogonal to any sequence
$f(T^nx)$, for almost every point $x \in X$ with $f \in L^2(X)$
(we assume $X$  is equipped with a probability
measure in this case).

The case of the orthogonality of the M\"{o}bius function to any nilsequence
$X=G/\Gamma$, where $G$ is a nilpotent Lie group, $\Gamma$ is
a lattice in $G$ and $T_g(\Gamma x)= \Gamma x g.)$ is covered by Green-Tao Theorem
\cite{Green-Tao}.
Recently, Bourgain-Sarnak and Ziegler in \cite{Bourgain-sarnak-Ziegler}
extend the Green-Tao result and establish that the M\"{o}bius function is orthogonal to
the horocycle flow. Indeed, applying Bourgain-Sarnak-Ziegler criterion, on can get
a simple proof of Green-Tao Theorem \cite{Ziegler}.
Roughly speaking, Bourgain-Sarnak-Ziegler criterion implies Conjecture
\eqref{sarnak-conj} is true when $\mu(n)$ is replaced by any  multiplicative
function with module less than 1 provided that the
mutual powers of the two dynamical systems are disjoint in the sens of Furstenberg.
It is well known that this latter property holds in the case of the spectral
disjointness of the mutual powers.
Therefore, the spectral disjointness of different primes powers implies the
disjointness of different primes  which ensures by Bourgain-Sarnak-Ziegler criterion
that the Sarnak Conjecture holds.

Exploiting this fact, el Abdalaoui-Lemanczyk-de-la-Rue obtain in the very recent work
\cite{elabdal-lem-de-la-rue} a new proof of Bourgain Theorem \cite{Bo} saying
that the rank one maps with bounded parameters are orthogonal to  M\"{o}bius.
Precisely, the authors extend Bourgain Theorem to a large class of rank one maps.

We should point out here that the  spectral disjointness of the mutual powers doesn't
holds in the case of Lebesgue spectrum.
\vspace{0.2cm}

There is a conjecture due to Elliott \cite{Elliott-C} saying that the spectral measure
of $\mu$ is exactly the Lebesgue measure up to a constant. As far as we know,
this conjecture is still open. Here we are able to prove that, under Sarnak's conjecture
 (plus one technical assumption)
the spectral measure of M\"obius function is equivalent
to the Lebesgue measure.

Moreover we give a new simple proof of Sarnak and Cellarosi-Sinai result
and assuming Elliott conjecture we establish that
the M\"obius function is orthogonal to any uniquely ergodic dynamical system
with singular spectrum.

Our proof uses a spectral approach based on the classical methods introduced
in \cite{Coquet-France} and intensively used to study the spectrum of arithmetical
$q$-multiplicative functions \cite{Kamae},\cite{Coquet}.
In \cite{Queffelec1}, \cite{Queffelec2},  M. Queffelec used the
standard method of Riesz products to obtain more results on the spectrum of
$q-$multiplicative functions. In a forthcoming paper \cite{elabdal-lem}, the first
author and M. Lema\'nczyk give a new proof of all these results based on the Bourgain
methods introduced in the context of generalized Riesz products associated to the
spectrum of rank one maps \cite{Bo2}. In addition, they establish the orthogonality of
M\"obius to $q-$multiplicative functions and, as a consequence, to any Morse
sequences \cite{AKL}.
\vspace{0.2cm}

\paragraph{\bf The random M\"obius function}
In this paper, following Ng \cite{Ng}, we  \enquote{simulate} randomly the behavior
of the values  of the M\"obius function in the following way.
Let $\pri$ be the set of prime numbers. A positive number $n$ is square-free
({\it {quadratfrei}}) if $p^2 \nmid n$ for every $p \in \pri$ $\big
($where $a \nmid b$ means $a$ does not divide $b$, and
$a,b \in \N$ $\big).$  Denote the set of all square-free numbers by $\Q$ and
$\omega(n)$ the number of distinct prime factors of $n$. Thus, for any $n \in \Q$,
we have $\mu(n)=1$ if $\omega(n)$ is even and $\mu(n)=-1$ if $\omega(n)$ is odd.
To simulate randomly this behavior
let $\epsilon_n$ be a sequence  of independent Rademacher random variables, indexed by
$n \in \Q$, that is
\begin{equation}\label{rademacher}
\P(\epsilon_n =1)=\P(\epsilon_n =-1)=\frac{1}{2} \qquad  \mbox{independently for each}\  n \in\Q.
\end{equation}
The random M\"obius function $\mu_{\rand}$ is then defined by
\begin{equation}\label{murand}
\mu_{\rand}(n)=\begin{cases}
\displaystyle \epsilon_n &{\rm {~~if~~}} n \in \Q \\
0 & n\in \mbox{otherwise} .
              \end{cases}
\end{equation}
In this work we prove that almost surely the random M\"obius function is orthogonal to
any topological dynamical system with zero entropy. Our main tool is based on a
concentration inequality due to Hoeffding and Azuma combined with a Borel-Cantelli
argument.

\begin{rem}
The probability that $\mu_{rand}  = \mu $ is zero.\medskip
\end{rem}

The rest of the paper is organized as follows: in section 2, we recall some basic
facts from the spectral analysis of dynamical systems, we give  the definition of
topological entropy and  state our main results.
In section 3, we prove our results on the deterministic M\"obius function
and finally in section 4 we prove our result for the random case.
In the Appendix we added some results that, though they  not
necessary for the proof (we only need a weaker version), may be
useful for some future generalizations.

\section{Main results and some basic facts
from spectral theory for dynamical systems}

A dynamical system is a pair $(X,T)$ where $X$ is a space
(equipped with a topology, a metric, or a probability measure) and $T$
is a measurable bijection $T:X\to X$.
In this paper we will fix  $X$ to be a {\em metric compact space} and $T$ a continuous function.
We will consider three types of dynamical systems: $(a)$
 dynamical system in a measurable setting ($(a)'$ weakly mixing or
$(a)''$ uniquely ergodic) and $(b)$ topological with zero entropy.
Our main problem will be to study {\em dynamical sequences}
$a_{n}=f (T^{n}x),$ or   {\em dynamical processes} $Y_{n}=f\circ T^{n},$
generated inside some dynamical system via a continuous function $f$.

\subsection{\bf Dynamical systems in a measurable setting}

We endow the space $X$ with a probability measure structure $(\A,P)$.
We also require  $T$ to be bimesurable and to preserve $P$,
i.e. $P(T^{-1}(A))=P(A)$, for every $A \in \A$. The dynamical system
is ergodic if the $T$-invariant set is trivial:
$P(T^{-1}(A) \triangle A)=0 \Longrightarrow P (A) \in \{0,1\}$.
 $T$ induces an
operator $U_T$ in $L^p(X)$ via $f \mapsto U_{T} (f)= f \circ T$
called Koopman operator.
For  $p=2$ this operator is unitary and its spectral
resolution induces  a spectral decomposition of $L^{2} (X)$ \cite{parry}:
\[
L^2(X)=\bigoplus_{n=0}^{+\infty}C(f_i) {\rm {~~and~~}}
\sigma_{f_1}>> \sigma_{f_2}>>\cdots
\]
where
\begin{itemize}
\item $\{f_{i} \}_{i=1}^{+\infty}$ is a family of functions in $L^{2} (X)$;
\item $C(f)\setdef \overline{\rm {span}}\{U_T^n(f): n \in \Z\}$ is the cyclic
space generated by $f \in L^2(X)$;
\item  $\sigma_f$ is the {\em spectral measure} on the circle generated by $f$
via the Bochner-Herglotz relation
\begin{equation}\label{fspmeasure}
\widehat{\sigma_f}(n)=<U_T^nf,f>=\int_X f \circ T^n(x) \overline{f}(x)dP(x);
\end{equation}
\item for any two measures on the circle $\alpha$ and $\beta$,
$\alpha >> \beta$ means $\beta $ is absolutely continuous  with respect to
$\alpha$: for any Borel set, $\alpha(A)=0 \Longrightarrow \beta(A)=0$.
The two measures $\alpha$ and  $\beta$ are equivalent if  and only if
$\alpha>>\beta$ and $\beta>>\alpha$. We will denote measure equivalence by
$\alpha \sim \beta$.
\end{itemize}
The spectral theorem ensures this spectral decomposition is unique up to
isomorphisms.
The {\em maximal spectral type} of $T$ is the equivalence class of the Borel
measure $\sigma_{f_1}$. The multiplicity function
$\M_{T} : \T \longrightarrow \{1,2,\cdots,\} \cup \{+\infty\}$ is defined
$\sigma_{f_1}$ a.e. and
\[
\M_T(z)=\ds \sum_{n=1}^{+\infty}\1_{Y_j}(z), \quad {\rm  where}, \ Y_1=\T \ {\rm and}\
Y_j={\rm{~supp~}}\frac{d\sigma_{f_j}}{d\sigma_{f_1}} \quad \forall j \geq 2.
\]
An integer $n \in \{1,2,\cdots,\} \cup \{+\infty\}$ is called an essential value of
$M_T$ if $\sigma_{f_1}\{ z \in \T : M_T(z)=n\}>0$. The multiplicity
is uniform or homogenous if there is only one essential value of $M_T$.
The essential supremum of $M_T$ is called the maximal  spectral multiplicity of $T$.
The map $T$
\begin{itemize}
\item has simple spectrum if $L^2(X)$ is reduced to a
single cyclic space;
\item has discrete spectrum if $L^2(X)$ has an
orthonormal basis consisting of eigenfunctions of $U_T$
(in this case $\sigma_{f_1}$ is a discrete measure);
\item has Lebesgue spectrum (resp. absolutely continuous,
singular spectrum) if
$\sigma_{f_1}$ is equivalent  (resp. absolutely
continuous, singular) to the Lebesgue measure.
\end{itemize}
The {\em reduced spectral type } of the dynamical system  is its spectral type on
the $L_0^2 (X)$ the space of square integrable functions with zero mean.
\begin{defn}\label{defsingular}
Two dynamical systems are called {\em {spectrally disjoint}} if their reduced
spectral types are mutually singular.
\end{defn}

In this paper we consider two types of  measurable dynamical systems:
\begin{itemize}
\item [$(a)'$]  $(X,T,\A)$ is {\em uniquely ergodic} if
 $X$ is a compact metric space, $T$ is a homeomorphism and there exists a
unique ergodic probability measure $P$ ;
\item [$(a)''$]
$(X,T,\A,P)$ is {\em uniquely weakly mixing} if  there exists a
unique ergodic probability measure $P$ and $\sigma_{f_{1}}$
is the sum of the Dirac  measure on zero and a continuous measure.
\end{itemize}
With the Jewett result \cite{Jewett} in mind, it is easy to see that we
still deal with a large class of dynamical systems.

\subsection{\bf Topological dynamical systems and topological entropy}
The pair  $(X,T)$ is called {\em  topological dynamical system} if
$X$ is a compact metric space and $T:X\to X$ is a homeomorphism.

There are different ways to define the topological entropy for such a system.
Here we use the one introduced by Bowen \cite{Bowen} based on
$(m,\epsilon)$-spanning sets.
For each $m \in \N$ we  define a new distance $d_m$ given by

\begin{equation}\label{mTdistance}
d_m(x,y)=\max_{0 \leq k \leq m-1}d(T^kx,T^ky), \qquad x,y\in X.
\end{equation}
Two points in $X$ are $\epsilon$-close with respect to the distance $d_m$
if their iterates under $T$ stay $\epsilon$-close up to time $m-1$.
Note that the definition of $d_m$ depends on the transformation $T$.
The open ball with respect to this metric

\begin{equation}\label{mTball}
B_{d_m}(x, \epsilon) = \{y\in X:~ {\rm ~{such~that~}} d(T^kx, T^ky) < \epsilon
\ \forall \  0 \leq k \leq m-1\}
\end{equation}
consist of all points whose trajectories up to time $m-1$ remain $\epsilon$-close
to the finite orbit segment
$\{x, Tx,\cdots, T^{m-1}x\}$.  Since $X$ is compact, for any $m\in \mathbb{N}$
and $\epsilon >0$, there exists a finite
set of points $R(m,\epsilon)\subset X$ of minimal cardinality $r(m,\epsilon)$

\begin{equation}\label{Rset}
R(m,\epsilon)= \left \{ x^{(m)}_{1},\dotsc, x^{(m)}_{r (m,\epsilon)} \right\}\subset X,
\end{equation}
 such that
\begin{equation}\label{covering}
X =\bigcup_{j=1}^{r (m,\epsilon)} B_{d_m}(x^{(m)}_{j},\epsilon).
\end{equation}
The cardinality $r(m,\epsilon)$ is then the minimal number of balls we need
to describe all possible segments of trajectory of length $m$.
The {\em topological entropy} of $(X,T)$ is defined by

\begin{equation}\label{top-entr}
h(T)=\lim_{\epsilon \longrightarrow 0}
\limsup_{m \longrightarrow +\infty}\frac{1}{m}\log\left[ r(m,\epsilon) \right].
\end{equation}
A nice account on the topological entropy may be found in \cite{ Walters},\cite{PYuri}.
With these definitions,  the dynamical system $(X,T)$ has {\em zero topological entropy}
if for any positive constant $\eta >0$ there exists a positive constant
$\epsilon_{0} (\eta)>0$ and a positive integer $m_{0} (\eta)>0$ such that
\begin{equation}\label{zerote}
r (m,\epsilon )< e^{m\eta } \qquad \forall \epsilon<\epsilon_{0} (\eta ),\
\forall m>m_{0} (\epsilon ).
\end{equation}
This property will be crucial to prove our result.

\subsection{\bf Spectral measure of a sequence}\label{sequence-spectral-measure}
The notion  of  spectral measures for sequences is introduced by Wiener in
his 1933 book \cite{Wiener}. More precisely, Wiener considers the space $S$ of
complex bounded sequences $g =(g_{n})_{n \in \N}$ such that
\begin{equation}\label{Sspace}
\lim_{N \longrightarrow +\infty}
\frac{1}{N}\sum_{n=0}^{N-1}g_{n+k}\overline{g}_{n}=F(k)
\end{equation}
exists for each integer $k \in \N$. The sequence $F (k)$ can be extended to negative
integers by setting
\[
F(-k)=\overline{F (k)}.
\]
It is well known that $F$ is positive definite on $\Z$ and therefore (by
Herglotz-Bochner theorem) there exists a unique positive finite measure $\sigma_g$ on
the circle $\T$ such that the Fourier
coefficients of $\sigma_g$ are given by the sequence $F$.
Formally, we have
\[
\widehat{\sigma_g}(k)\stackrel{\rm{def}}{=}
\int_{\T} e^{-ikx} d\sigma_{g }(x) = F(k).
\]
The measure $\sigma_g $ is called the {\em spectral measure of the sequence $g$}.
\medskip

The orthogonality issue of the M\"{o}bius function
may  be connected to the spectral analysis of
the sequence $\mu (n)$ on one side  and the nature of the spectral type of the dynamical sequence
$g_{n}=f (T^{n}x)$ on the other.
Indeed, there exists a natural connection between the spectral measure
and the spectral type of the dynamical system.
For a uniquely ergodic system we have the following result.
\begin{lemm}\label{sp-f}
Let $(X,\A,P,T)$ be a uniquely ergodic topological dynamical system. Then, for
any $f \in C(X)$, for every $x \in X$, the sequence
$g_{n}=f(T^nx)$ belongs to the Wiener space $S$ and its spectral measure is exactly
the spectral measure of the function $f$ given by
\[
\widehat{\sigma}_f(k)=<U^{k}_T(f),f>=\int f \circ T^{k}(x). \overline{f(x)} dP(x),
\]
where $U_T$ is a unitary operator on $L^2(X)$ defined by $f \mapsto U_{T} (f)= f \circ T.$
\end{lemm}
\begin{proof} Let $g_{n}= f(T^nx)$. Then, for any $N>1$,
\[
\frac{1}{N}\sum_{n=0}^{N-1}g_{n+k}\overline{g_{n}}
 =  \frac{1}{N}\sum_{n=0}^{N-1}f (T^{n+k}x)\overline{f (T^{n})}
  =  \frac{1}{N}\sum_{n=0}^{N-1} \left[  (f\circ T^{k}) \cdot \overline{f}\right] (T^{n}x).
\]
By the unique ergodicity of the system, the right-hand side converges
\[
\frac{1}{N}\sum_{n=0}^{N-1} \left[  (f\circ T^{k}) \cdot \overline{f}\right] (T^{n}x)
\tend{N}{\infty} \int (f \circ T^k)(y) \cdot \overline{f(y)} dP(y) =
\widehat{\sigma_{f}}{(k)}.
\]
Then the sequence $(g_{n})$ belongs to the Wiener space $S$ and
its spectral measure coincides with the spectral
measure of $f$
\[
 \widehat{\sigma}_{g } (k)= \lim_{N \longrightarrow +\infty}
\frac{1}{N}\sum_{n=0}^{N-1}g_{n+k}\overline{g_{n}} =
\int (f \circ T^k)(y) \cdot \overline{f(y)} dP(y)
=\widehat{\sigma_{f}}{(k)}.
\]
The proof of the lemma is complete.
\end{proof}

\begin{rem}\label{Birk}
In the case of any ergodic system (not necessarily uniquely ergodic), the
result of Lemma \ref{sp-f}  above still holds for almost every $x\in X$,
by Birkhoff theorem.
\end{rem}

Therefore,  the orthogonality issue of the M\"{o}bius function
seems to be related  to the computation of
the M\"{o}bius spectral measure or, in other words, the
self-correlations of the M\"obius function.
On the other hand, the weaker form of the Elliott conjecture \cite{Elliott-C}  implies that
the spectral measure of the M\"{o}bius function is up to a constant
the Lebesgue measure on the circle, i.e.
\begin{conj}[of Elliott]\cite{Elliott-C}
 \begin{eqnarray}\label{Elliot}
  \lim_{N \longrightarrow +\infty} \frac{1}{N}\sum_{n=1}^{N}\mu(n)\mu(n+h)=\begin{cases}
0 &{\rm {~~if~~}} h \neq 0 \\
\displaystyle \frac{6}{\pi^ 2} &{\rm {~~if~not~}}.
\end{cases}
 \end{eqnarray}
\end{conj}
P.D.T.A. Elliott writes in his 1994's AMS Memoirs that ``even the simple particular
cases of the correlation (when $h=1$ in \eqref{Elliot}) are not well understood.
Almost surely the M\"obius function satisfies \eqref{Elliot} in this case, but at
the moment  we are unable to prove it.''
\begin{rem} {\it (a)} \eqref{Elliot} implies that the sequence $\mu(n)$ belongs to
the Wiener space $S$ and its spectral measure is exactly (up to a normalization
constant) the Lebesgue measure on the circle.

 {\it (b)} An alternative way to state Elliott conjecture is the following
\cite{Elliott-C}. Let $a,b,A$ and $B$ be integers for
which $aB \neq Ab$. Then
\[
\frac{1}{N}\sum_{n=1}^{N}\mu(an+b)\mu(An+B)\tend{N}{+\infty}0,
\]
\end{rem}

Elliott conjecture may be seen as a consequence of
the $L^1$-flatness of the trigonometric polynomials with
M\"obius coefficients.
A sequence of polynomials $P_{n} (\theta)$  is  $L^1$-flat if
\[
\left\|  \frac{|P_{n}|}{\|P_n\|_{2} }-1   \right\|_{1} \tend{n}{\infty} 0.
\]
Suppose the following sequence of polynomials   with
M\"obius coefficients
\[
P_{n} (\theta)= \frac{1}{\sqrt{n}} \sum_{j=1}^{n} \mu (j) e^{ij\theta}
\]
satisfies
\[
 \int_{0}^{2\pi }
\left|\frac{ |P_{n} (\theta ) |^{2} }{\|P_{n}\|_{2}^{2}  } -1 \right |
\frac{d\theta }{2\pi }
\tend{N}{+\infty}0.
\]
where the normalization constant
\[
\|P_{n}\|_{2}^{2} =\int_{0}^{2\pi } |P_{n} (\theta ) |^{2} \frac{d\theta }{2\pi }
= \frac{1}{n} \sum_{j=1}^{n} \mu^{2} (j)\tend{N}{+\infty} \frac{6 }{\pi^{2} } >0
\]
corresponds to the fraction of square free integers in the interval $[0,n]$.
Then this sequence is  $L^1$-flat and the spectral measure for $\mu $ is
the Lebesgue measure. Indeed
\begin{align*}
\left|\frac{1}{n} \sum_{j=1}^{n} \mu (j) \mu (j+k) \right | &=
 \left|\int_{0}^{2\pi } e^{-ik\theta } |P_{n} (\theta ) |^{2}
\frac{d\theta }{2\pi }\right |
\cr
&\leq \|P_{n}\|_{2}^{2} \left[
\int_{0}^{2\pi }
\left|\frac{ |P_{n} (\theta ) |^{2} }{\|P_{n}\|_{2}^{2}  } -1 \right |
 + \left|\int_{0}^{2\pi } e^{-ik\theta } \frac{d\theta }{2\pi } \right | \right]
\end{align*}
Then
\[
\lim_{n\to\infty} \frac{1}{n} \sum_{j=1}^{n-1} \mu (j) \mu (j+k) = 0 \quad \forall
k\neq 0.
\]

In his lectures \cite{sarnak}  P. Sarnak constructs a uniquely ergodic topological dynamical system
called the {\em M\"obius flow}. For this dynamical system the M\"obius function is
a generic point and the topological entropy of this system is positive.\medskip

\noindent{} We are now able to state our main results.

\begin{thm}[{\bf  Main result 1}]\label{th4} Assume  the M\"obius function is
orthogonal to any singular dynamical process $Y_{n}$ on a compact metric space\footnote{Note that the notion of singular process gives an alternative definition of deterministic sequence which is in general not equivalent to the notion of 
deterministic  sequence based on the topological entropy.},
and in the M\"obius flow
conditional expectation preserves the space of continuous functions.  More precisely
we assume
\begin{equation}
\frac{1}{N}\sum_{n=1}^{N} Y_{n} (w)\  \mu(n) \  \tend{N}{+\infty}0\qquad
\forall w\in X.
\end{equation}
Then the spectral measure of $\mu (n)$ is absolutely continuous with
log integrable  density $\rho$.
\end{thm}

\begin{thm}[{\bf  Main result 2}]\label{wmt} Assume that Elliott conjecture
\eqref{Elliot} holds.
Then the M\"{o}bius function is orthogonal to any uniquely ergodic dynamical system with
singular spectrum. Precisely
\begin{equation}
\frac{1}{N}\sum_{n=1}^{N}f(T^nx)\  \mu(n) \  \tend{N}{+\infty}0\qquad
\forall f\in C_{0} (X), \forall x\in X,
\end{equation}
where $C_{0} (X)$ is the set of continuous functions on $X$ with zero mean.
\end{thm}

\begin{thm}[{\bf  Main result 3: Sarnak and Cellarosi-Sinai}]\label{cs}
The sequence $\mu^2(n)$   generated by the square of the M\"obius function
is spectrally orthogonal to the measure of any uniquely weakly mixing dynamical system
$f (T^{n}x)$.  More precisely
\begin{equation}
\frac{1}{N}\sum_{n=1}^{N}f(T^nx)\  \mu^{2}(n) \  \tend{N}{+\infty}0\qquad
\forall f\in C_{0} (X), \forall x\in X,
\end{equation}
where $C_{0} (X)$ is the set of continuous functions on $X$ with zero mean.
\end{thm}

\begin{rem}
Theorem \ref{cs} is essentially due to Sarnak \cite{sarnak} and
Cellarosi-Sinai \cite{Cellarosi-Sinai}. Indeed, in
\cite{sarnak} and \cite{Cellarosi-Sinai} the authors construct a dynamical system
associated to the sequence $\mu^2(n)$.
The dynamical system obtained is isomorphic to the translations on a compact
Abelian group.
\end{rem}

\begin{thm}[{\bf Main result 4}]\label{m} The random M\"obius function
is orthogonal to any topological dynamical
system with zero topological entropy.
More precisely, let $(X,T)$ be a topological dynamical
system with zero topological entropy and $C (X)$ the set of continuous
complex valued functions on $X$.
Then there exists a subset $\Omega'\subset \Omega$ independent
of $x$ and $f$ such that $\mathbb{P}(\Omega')=1$ and
\begin{equation}\label{eqm1}
\frac{1}{N}\sum_{n=1}^{N}f(T^nx)\  \mu_{\rand}(n) (\omega)\  \tend{N}{+\infty}0\qquad
\forall \omega \in \Omega',\ \forall x\in X,\ \forall f\in C(X).
\end{equation}
\end{thm}


\section{Spectral disjointness for the M\"obius function}


In this section we consider the (non random) M\"obius
function and  prove Theorem  \ref{cs},   \ref{wmt} and \ref{th4}.
Both proofs use the notion of affinity. Here we consider
only uniquely ergodic dynamical systems so the spectral
measure of the corresponding sequence is well defined.
We will see in subsection \ref{smformu2} below that the
spectral measure of $\mu^{2} (n)$ is
also well defined.
We start by giving a few preliminary results and definitions.

\subsection{\bf Affinity and correlations}
The affinity between two finite measures is defined by the integral of the corresponding
geometric mean. Is is introduced and studied in a series of papers by Matusita
\cite{Matusita1},\cite{Matusita2},\cite{Matusita3} and  it is also called
Bahattacharyya coefficient \cite{Bhatta}. It has been widely used in statistics
literature as a useful tool to quantify the similarity between two
probability distributions.  Ii is symmetric in distributions and
has direct relationships with error probability when classification or discrimination
is concerned.

Let $\M_1(\T)$ be a set of probability measures on the circle $\T$ and
$\mu,\nu \in \M_1(\T)$ two mesures in this space. There
exists a probability measure $\lambda$ such that $\mu$ and $\nu$ are absolutely
continuous with respect to $\lambda$ (take for example $\lambda=\frac{\mu+\nu}2$).
Then the {\em affinity} between $\mu$ and $\nu$ is defined by
\begin{equation}\label{affinity}
G(\mu,\nu)=\ds \int \sqrt{\frac{d\mu}{d\lambda}. \frac{d\nu}{d\lambda}} d\lambda.
\end{equation}
This definition does not depend on $\lambda$.
Affinity is related to the Hellinger distance as it can be defined as
\[
H(\mu,\nu)=\sqrt{2(1-G(\mu,\nu))}.
\]
Note that $G (\mu ,\nu )$ satisfies (by Cauchy-Schwarz inequality)
\[
 0 \leq G(\mu,\nu) \leq  1.
\]
\begin{rem}
The definition of affinity can be extended to any pair of
 positive non trivial finite measures. Indeed any such measure becomes a
probability measure if we divide by the (positive) normalization factor.
\end{rem}
It is an easy exercise to see that $G(\mu,\nu)=0$ if and only if $\mu$ and $\nu$
are mutually singular (denoted by $\mu \bot \nu$):
this means $\mu$ assigns measure zero to every set to which $\nu$ assigns a
positive probability, and vice versa. Similarly, $G(\mu,\nu)=1$ holds if
and only if $\mu$ and $\nu$ are equivalent: $\mu\ll \nu$ and $\nu\ll \mu$.
Affinity can be used to compare sequences of measures via the
following theorem. The proof may be found
in \cite{Coquet-France}.
\begin{thm}[Coquet-Kamae-Mand\`es-France  \cite{Coquet-France}]\label{coquet-france}
Let $(P_n)$ and $(Q_n)$ be two sequences of probability measures
on the circle weakly converging to the probability measures $P$ and $Q$
respectively. Then
\begin{equation}\label{limsup}
\limsup_{n \longrightarrow +\infty } G(P_n,Q_n) \leq G(P,Q).
\end{equation}
\end{thm}
As in the case of the affinity, this result can be generalized
to any sequence of positive non trivial finite measures $P_{n}$, $Q_{n}$
converging weakly to two positive non trivial finite measures $P$ and $Q$.

We want to use the affinity and  Theorem \ref{coquet-france}  above to estimate
the orthogonality properties of  pairs of sequences in the Wiener space
$S$ (defined in subsection \ref{sequence-spectral-measure}).
To do that we will need to replace the sequence of Fourier coefficients
$\frac{1}{n}\sum_{j=0}^{n-1} g_{j+k}\overline{g_{j}}$ by
a sequence of finite positive measures on the circle.
For any $g\in S$, we introduce  the sequence of functions
\begin{equation}\label{rhodef}
d\sigma_{g,n} (x)= \rho_{g,n} (x)\frac{dx}{2\pi },\qquad \mbox{where} \
\rho_{g,n} (x) = \left|\frac{1}{\sqrt{n}}\sum_{j=0}^{n-1}g_{j}e^{ijx}\right|^2.
\end{equation}
With this definition $\sigma_{g,n}$ defines a finite positive measure on the
circle.
Moreover  we have the relation
\[
\frac{1}{n}\sum_{j=0}^{n-1} g_{j+k}\overline{g_{j}}=
\int_{0}^{2\pi } e^{-ikx} d\sigma_{g,n} (x) \ + \ \Delta_{n,k}  =
\widehat{\sigma }_{g,n} (k)+ \Delta_{n,k},
\]
where
\[
\left| \Delta_{n,k}  \right|=
\left| \frac{1}{n}\sum_{j=n-k}^{n-1} g_{j+k}\overline{g_{j}}  \right| \leq \
\frac{k}{n} \sup_j |g_{j}|^{2} \longrightarrow_{n\to\infty} 0.
\]
Taking the limit we have
\[
\widehat{\sigma }_{g} (k) = \lim_{n\to\infty}
\frac{1}{n}\sum_{j=0}^{n-1} g_{j+k}\overline{g_{j}}= \lim_{n\to\infty}
\widehat{\sigma }_{g,n} (k)
\]
so the sequence of measures $\sigma_{g,n}$ converges weakly to $\sigma_{g}$.

To prove our theorems we will need the following result.
\begin{Cor}\cite{Bellow-Losert}\label{BL} Let $g = (g_{n})_{n\in \N},
h= (h_{n})_{n\in \N} \in S$ two non trivial sequences i.e.
$\widehat{\sigma}_{g} (0)>0$ and $\widehat{\sigma}_{h} (0)>0$. Then
\begin{equation}
\limsup_{n\to \infty} \left | \frac{1}{n}  \sum_{j=1}^{n}g_{j}\overline{h}_{j}\right|
\leq G(\sigma_g ,\sigma_h).
\end{equation}
\end{Cor}
\begin{proof} Take
$P_n = \sigma_{g,n}$ and $Q_n = \sigma_{h,n}$ (see \eqref{rhodef}). Then
\[
 \left | \frac{1}{n}  \sum_{j=1}^{n}g_{j}\overline{h}_{j}\right| =
 \left | \frac{1}{n} \int_{\T}   \sum_{j,k=1}^{n} g_{j} e^{ijx}
\overline{h}_{k} e^{-ikx} dx \right|
\leq  G (P_{n}, Q_{n}).
\]
Applying  \eqref{limsup} the result follows.
\end{proof}

\begin{rem}
It may be possible that the bounded sequence $g_{n}$ does not belong to the
Wiener space $S$ (see eq. \eqref{Sspace}) but
we can always extract a subsequence $(n_r)$ in \eqref{Sspace} such that
\[
 \lim_{r \rightarrow \infty}
\frac{1}{n_r}\sum_{j=0}^{n_r-1}g_{j+k} \overline{g_{j}}
\]
exists for each $k\in \N$.
In fact, define the sequence of finite positive measures
$(\sigma_{g,n})_{n \in \N}$ on the torus defined in \eqref{rhodef}.
These measures are all finite and $\sigma_{g,n} (\T)\leq  \|g\|_\infty^{2}
= \sup_j |g_{j}|^{2}$ $\forall n$. Therefore they all belong to the subset of
measures on the circle $\mathcal{B} (0, \|g\|_\infty^{2} )$.
This subset is compact so there exists a
subsequence $(n_r)$ such that
the sequence of probability measures $(\sigma_{g ,n_r})_{r \in \N}$
converge weakly to some probability measure $\sigma_{g,(n_r)}$.
The measure $\sigma_{g ,(n_r)}$ is called the spectral
measure of the sequence $g$ along the subsequence $(n_r)$.
\end{rem}

Theorem  \ref{coquet-france} can be extended to a lower bound
on the absolutely continuous part of the spectral measure of
a given sequence in the Wiener space. More precisely we have the
following proposition.
\begin{Prop}\label{Bourgain} Let $g \in S$ a non trivial sequence, i.e.
$\widehat{\sigma}_{g} (0)>0$. Then
\begin{equation}\label{Bourgain-eq}
\limsup_{n \longrightarrow +\infty}
\int_{\T} \left|\frac{1}{\sqrt{n}}\sum_{j=0}^{n-1}g_{j} e^{ijx}\right| dx
\leq \int \sqrt{\frac{d\sigma_{g,a} }{dx}} dx
\end{equation}
where  $\displaystyle \frac{d\sigma_{g,a}}{dx}$ is the Radon-Nikodym derivative
of the Lebesgue component of $\sigma_g$.
\end{Prop}
\begin{proof}
Let $n_{k}$ be a subsequence such that
\[
\limsup_{n \longrightarrow +\infty}
\int_{\T} \left|\frac{1}{\sqrt{n}}\sum_{j=0}^{n-1}g_{j} e^{ijx}\right| dx =
\lim_{k} \int_{\T} \left|\frac{1}{\sqrt{n_{k}}}\sum_{j=0}^{n_{k}-1}g_{j} e^{ijx}\right| dx.
\]
Put
\[
\xi_{k}= \left|\frac{1}{\sqrt{n_{k}}}\sum_{j=0}^{n_{k}-1}g_{j} e^{ijx}\right| dx.
\]
Then there exists a subsequence of $\xi_{k}$ (that we still denote by $n_{k}$ to avoid
heavy notation) such that $\xi_{k}$ converges weakly to some positive finite measure $\xi$.
We will show later that for any Borel  set $A$ of $\T$ we have
\begin{equation}\label{etoile}
\xi (A)\leq  \sqrt{\sigma_{g} (A)} \sqrt{\lambda  (A)}.
\end{equation}
Now let $E$ be a Borel set such that
\[
\lambda (E)=1\quad  \mbox{and}  \quad  \sigma_{g,s} (E)=0,
\]
where $\sigma_{g,s} $ is the singular part of $\sigma $ with respect to the
Lebesgue measure. Then by \eqref{etoile} $\xi$ is absolutely  continuous
and  for any Borel set $A\subseteq E$ we have
\[
\int_{A} \frac{d\xi}{dx } dx \leq  \sqrt{ \int_{A} \frac{d\sigma_{g,a} }{dx } dx}
  \sqrt{ \int_{A} dx }
\]
By a Martingale argument we deduce that for almost all $x$
\[
\frac{d\xi}{d\lambda } (x) \leq  \sqrt{ \frac{d\sigma_{g,a} }{dx }} (x)
\]
From this the proof of the proposition follows.
It remains to prove \eqref{etoile}.
Let $w$ be a positive continuous function. Then by Cauchy-Schwarz inequality
\[
\int_{\T} w (x) \left|\frac{1}{\sqrt{n_{k}}}\sum_{j=0}^{n_{k}-1}g_{j} e^{ijx}\right| dx
\leq   \left[  \int_{\T} w (x) \left|\frac{1}{\sqrt{n_{k}}}\sum_{j=0}^{n_{k}-1}g_{j} e^{ijx}\right|^{2} dx
 \right]^{1/2}  \left[ \int_{\T} w (x) dx \right]^{1/2}
\]
Letting $k$ go to infinity we get
\[
\int_{\T} w (x) d\xi \leq  \left[ \int_{\T} w (x) d\sigma_{g} \right]^{1/2}
\left[ \int_{\T} w (x) dx\right]^{1/2}
\]
Hence, by the density of subspace of continuous functions in $L^{2} (\sigma_{g}+dx+\xi )$ the claim
follows. The proof of the proposition in then complete.
\end{proof}

\subsection{\bf The spectral measure for $\mu^{2} (n)$}\label{smformu2}
Before starting the proof of the two main theorems, we need to
introduce some additional results on the spectrum of the
sequence $\mu^{2} (n)$.
\begin{thm}\label{mirsky}[Mirsky, 1948 \cite{Mirsky}]
As $N \longrightarrow{\infty}$, we have
\[
\sum_{n=0}^{N} \mu^2(n)\mu^2(n+k)= N\left[\prod_{p \in \pri} \left(1-\frac{2}{p^{2}} \right)
\ \prod_{p^2 | k} \left(1+\frac{1}{p^2-2} \right) \right] \
+\ O\left(N^{\frac{2}{3}}\log^{\frac{4}{3}}(N)\right),
\]
where the ${O}$-constant may depend on $k$, and we assume the empty product gives
contribution 1 (this happens when $k$ is square free).
\end{thm}
From Mirsky Theorem \ref{mirsky} we deduce that the spectral measure of
the square of the M\"obius function $\sigma_{\mu^2}$ is  discrete.
Precisely we have the following result.
\begin{Cor}\label{cor4.5} The spectral measure of the square of the M\"obius function
$\sigma_{\mu^2}$ is given by
\begin{equation}\label{sigmamusq}
\sigma_{\mu^2}=
\prod_{p \in \pri}\left(1-\frac{2}{p^2}\right)
\left( \delta_{0} + \sum_{l=1}^{\infty}
\sum_{\substack{d\in \mathcal{Q} \\ \omega (d)=l}}
\left \{ \frac{1}{d^2}\left[ \prod_{p\in D (d)} \frac{1}{(p^2-2)}  \right]
\sum_{j=0}^{d^2-1}\delta_{e^{\frac{2ij\pi}{d^2}}}  \right\}
\right)
\end{equation}
where $D (d)$ is the set of distinct primes in the decomposition of
$d\in \mathcal{Q}$,
$d= \prod_{p\in D (d)}p$ and $\omega (d)$ is the cardinal of $D (d)$.
\end{Cor}
\begin{proof}It is easy to see that for any $d,k \geq 1$ we have
\[
\sum_{j=0}^{d-1}e^{\frac{2\pi i j k}{d} }=\begin{cases}
 0 &{\rm {~if~}} d \nmid k \\
 d & {\rm {~if~}} d \mid k.
\end{cases}
\]
To see it put $x=e^{\frac{2\pi ik}{d}}$ and observe that $x=1$ if and only if $d \mid k$.
Combining this fact with Mirsky Theorem \ref{mirsky} we have, for any $k \geq 0$,
\[
\frac{1}{N}\sum_{n=1}^{N}\mu^2(n+k)\mu^2(n)
\tend{n}{\infty} \sigma_{\mu^{2}}
\]
where $\sigma_{\mu^{2}}$ is the finite measure on the circle given above \eqref{sigmamusq}.
 This completes the proof.
\end{proof}

\subsection{\bf Proof of Main results 1, 2 and 3}

\begin{proof}[{\bf Proof of  Main result 3 (Theorem \ref{cs})}]
The proof is a direct consequence of Mirsky Theorem \ref{mirsky}
combined with Corollary \ref{cor4.5} above.

Let $(X,\A,P,T)$ be a
uniquely weakly mixing dynamical system. Thus, for any continuous function with zero mean it
is well known from the classical Wiener Theorem \cite{Nadkarni} that the spectral
measure of $f$ is continuous. Therefore it is orthogonal
to any discrete measure. By lemma \ref{sp-f}, this is true also for the
spectral measure $\sigma_{g}$ of the sequence $g_{n}=f (T^{n}x)$ for  every $x$.
 But  the spectral measure of $\mu^2(n)$ is indeed discrete
(by corollary \ref{cor4.5} above), so
\[
G (\sigma_{\mu^{2}},\sigma_{g})=0.
\]
Now, applying  \ref{BL} we have
\[
\frac{1}{N}\sum_{n=1}^{N}f(T^nx)\mu^2(n) \tend{N}{\infty}0\  ,\forall
f\in C_{0} (X), \forall   x\in X.
\]
This achieves the proof of the Theorem. \qed
\end{proof}

\begin{rem}
This result is similar to Wiener-Wintner Theorem \cite{Wiener-Wintner}.
\end{rem}

\begin{proof}[{\bf Proof of Main result 2 (Theorem \ref{wmt})}]

Let $(X,\mathcal{A},T,P)$ be a uniquely ergodic dynamical system and suppose
 \eqref{Elliot} holds. Then the sequence $\mu (n)$ belongs to the Wiener space $S$
and is non trivial $\widehat{\sigma}_{\mu} (0)>0$. By Lemma \ref{sp-f}  the sequence
$f (T^{n}x)$ also belongs to the Wiener space, it is non trivial since
 $\hat{\sigma }_{f} (0)>0$ unless $f$ is zero  everywhere,  and its spectral measure
coincides with the spectral measure of $f$. Then
Corollary \ref{BL} gives
\[
\limsup\Big|\frac1{N}\sum_{n=1}^{N}f(T^nx) \mu(n)\Big| \leq G(\sigma_f,\sigma_{\mu})
\]
for any continuous function $f$, for any $x \in X$.
But under our assumptions the spectral measure of $\mu$ is the Lebesgue measure
and $\sigma_f$ is singular, therefore
\[
G(\sigma_f,\sigma_{\mu})=0.
\]
This implies that
\[
\frac{1}{N}\sum_{n=1}^{N}f(T^nx) \mu(n)\  \tend{N}{+\infty}\ 0.
\]
The proof of the Theorem is complete.
\end{proof}

To prove this theorem it would be enough
to assume that the spectral measure of $\mu$ is absolutely continuous with
respect to the Lebesgue measure (Elliott conjecture implies  it is exactly equal to it).
Theorem \ref{th4} below provides the reverse statement.

\begin{proof}[{\bf Proof of Main result 1 (Theorem \ref{th4})}]
In his Lectures \cite{sarnak} Sarnak constructs a dynamical system with positive entropy,
called M\"obius dynamical system, for which $\mu (n) $ is a generic point.
Let us denote this  dynamical system by $(X,\B,\P,S)$.
Therefore, there exists a continuous function such that
$$\mu(n) =f(S^nx),$$
\noindent where $x$ is a generic point.
Let $Y_{n}=f\circ S^{n}$. By Wold decomposition we can decompose $Y_{n}$
as
\[
Y_{n}  = Y_{n}^{r}  +  Y_{n}^{s}
\]
where $Y_{n}^{r}$ (resp. $Y_{n}^{s}$) is a regular  (resp. singular) process.
By our assumptions, $Y_{0}^{s}$ (resp. $Y_{0}^{r}$) is a continuous function and
\[
\frac{1}{N}\sum_{n=1}^{N}\mu (n) Y_{n}^{s} (x) \tendn 0
\]
hence,
\[
\frac{1}{N}\sum_{n=1}^{N}  [Y_{n}^{s}(x)]^{2}  \tendn 0
\]
since
\[
\frac{1}{N}\sum_{n=1}^{N}  Y_{n}^{s}(x)  Y_{n}^{r}(x)  \tendn \int Y_{0}^{s} Y_{0}^{r} d\P=0
\]
by the ergodic theorem. Therefore we get (again by the ergodic theorem)
\[
\int  [Y_{0}^{s}(x)]^{2} d\P=0.
\]
This means that $Y_{0}^{s}=0$. We conclude that the spectral  measure of $\mu (n)$
is absolutely continuous with density $\rho $ and $\ln |\rho |$ is
integrable  \cite[pp.28]{DaCa}. This completes the proof.
\end{proof}


\section{Spectral disjointness for the random M\"obius function}


We consider now the random version of the M\"obius function  defined in \eqref{murand}.
In order to prove  Theorem \ref{m} we need to find a subset
$\Omega'\subset \Omega$ such that $\mathbb{P}(\Omega')=1$ and
\begin{equation}
\frac{1}{N}\sum_{n=0}^{N-1} f (T^{n}x) \, \mu_{rand}(n) (w) \rightarrow_{N\to\infty} 0 \qquad
\forall w\in \Omega',\ \forall x\in X,\ \forall f\in  C(X).
\end{equation}
The proof is based on a Borel-Cantelli argument.
The main difficulty is to construct a set $\Omega'$ that is independent of $x$ and $f$.
Actually $C_{0}(X)$ is compact so it is not hard to get rid of the $f$ dependence.
On the other hand, the $x$ dependence is quite delicate. Though the space $X$ is compact
the orbits generated by  $T^{n}x$ may be too complicated to be controlled.
The key to solve the problem is to replace $f(T^{n}x) \, \mu_{rand}(n)$ in the
sum by an average over a long segment of the orbit starting at $T^{n}x$
\[
 f (T^{n}x) \, \mu_{rand}(n)\  \rightarrow
\ \frac{1}{m}\sum_{j=0}^{m-1}  f (T^{n+j}x) \, \mu_{rand}(n+j).
\]
The number of orbits can then be bounded using the zero topological entropy.

The rest of this section is devoted to the proof of the theorem.
We start by recalling a few basic tools and proving some preliminary results.
Finally in the last subsection we  give the proof of  Theorem \ref{m}.

\subsection{\bf Some preliminary results}

%
\begin{lemm}[Borel-Cantelli Lemma]\label{Borel-Cantelli Lemma}
Let $(X,\B,\nu)$  be a measure space and $E_n \subset \B$ be a countable collection
of measurable sets. We have
 \begin{equation}
   \sum_{n}\nu(E_n) < \infty \quad \Rightarrow\quad
   \nu(\limsup_{n \longrightarrow +\infty}E_n)=0,
\end{equation}
where
\begin{equation}
\limsup_{n \longrightarrow +\infty}E_n= \bigcap_n\bigcup_{k \geq n }E_k.
\end{equation}
\end{lemm}
This is a classical result in probability theory. In the following
we will also need the following concentration inequality.

\begin{lemm}\label{conc-ineq}
Let $x_{1},\dotsc ,x_{m}$ be $m$ independent random variables with
symmetric distribution $\mathbb{P} (x_{j})=\mathbb{P} (-x_{j})$,
$|x_{j}|\leq c_{j}$ for some positive constant $c_{j}$ and $\mathbb{E} (x_{j})=0$
$\forall j$. Then we have
\begin{equation}\label{cin}
\mathbb{P} \left( \left|  \sum_{j=1}^{m} x_{j}\right |> t \right) \leq \
2 e^{-\frac{t^{2}}{2\sum_{j=1}^{m} c_{j}^{2}}}
\end{equation}
\end{lemm}

\begin{proof}
By Markov inequality
\begin{align}
\mathbb{P} \left( \left|  \sum_{j=1}^{m} x_{j}\right |> t \right) &\leq \
e^{-\lambda t} \mathbb{E}  \left( e^{ \lambda  \left|  \sum_{j=1}^{m} x_{j}\right | } \right)
\quad \forall \lambda >0\cr
& \leq e^{-\lambda t} \mathbb{E}  \left( e^{ \lambda \sum_{j=1}^{m} x_{j} }+
 e^{- \lambda \sum_{j=1}^{m} x_{j} } \right)
\end{align}
Since the $x_{j}$ are independent
\begin{equation*}
\mathbb{E}  \left( e^{ \pm\lambda \sum_{j=1}^{m} x_{j} } \right) =
\prod_{j=1}^{m}\mathbb{E}  \left( e^{ \pm \lambda x_{j} } \right) =
\prod_{j=1}^{m}\mathbb{E}  \left( \cosh( \lambda x_{j})  \right)
 \leq  \prod_{j=1}^{m} \cosh(\lambda  c_{j})
\leq   \prod_{j=1}^{m} e^{\frac{(\lambda c_{j})^{2}}{2}}
\end{equation*}
where we used the symmetry of the distribution. Then
\begin{equation}
\mathbb{P} \left( \left|  \sum_{j=1}^{m} x_{j}\right |> t \right) \leq \
2 e^{\sum_{j=1}^{m}\frac{(\lambda c_{j})^{2}}{2}} e^{-\lambda t} \leq \
2 e^{- \frac{t^{2}}{2\sum_{j=1}^{m}c_{j}^{2}}}
\end{equation}
where in the last line we replaced $\lambda =t/\sum_{j=1}^{m}c_{j}^{2}$.
\end{proof}

\begin{rem}
 This phenomenon of the concentration of probability measure can be interpreted as
the absence of randomness for a large values.
The bound \eqref{cin} above is a special case of a more general concentration
inequality due to Hoeffding and Azuma \cite{Azuma},\cite{Mcdiarmid}.
Indeed it  can be generalized to the case of a sum of martingale differences.
The proof is given in appendix \ref{appA}.
\end{rem}

\begin{lemm}\label{entropy}
Let $X$ be a compact metric space, $T$ a continuous map on $X$, $f$ a complex valued
continuous function on $X$. Assume the topological entropy of the dynamical system
$(X,T)$ is zero.
We introduce the $m$-sequence
\begin{equation}\label{mseq}
\xi_{m} (T^{n}x)= \left(f (T^{n}x),f (T^{n+1}x),\dotsc ,f (T^{n+m}x) \right).
\end{equation}
This is the segment of the dynamical sequence generated by $m$ iterations
starting at $T^{n}x$.
Then
\begin{itemize}
\item for any $\delta>0$ there exists a constant $\varepsilon_{0} (\delta)$ such that
for any $\varepsilon<\varepsilon_{0} (\delta )$ we can find a minimal set of
points
\begin{equation}
R (m,\varepsilon )= \left(x_{1}^{(m)},\dotsc ,x_{r(m,\varepsilon )}^{{m}} \right)
\end{equation}
such that each $m$-sequence $\xi_{m} (T^{n}x)$ is localized in a $\delta$-neighborhood
of exactly
one of these points. More precisely for any $n,x$ there exists an integer
$1\leq j_{n,x}\leq r (m,\varepsilon )$
such that
\begin{equation}\label{dist}
\left| f (T^{n+j}x) -   f (T^{j}x^{(m)}_{j_{n,x}}) \right|\leq \delta \quad
\forall j=0,\dotsc ,m-1.
\end{equation}

\item The number of points $r (m,\varepsilon )$ we need in order to localize all
$m$-sequences does not grow too fast with $m$. Precisely, for any $\eta >0$
there exists a constant $\varepsilon_{1} (\eta )$ and an integer
$M_{0} (\eta )>0$ such that
\begin{equation}\label{entrbound}
r (m,\varepsilon )< e^{m\eta }\qquad \forall m> M_{0} (\eta ),\
\forall \varepsilon < \varepsilon_{1} (\eta ).
\end{equation}

\end{itemize}

\end{lemm}
\begin{proof}
The set $X$ being compact, for each choice of $\varepsilon $ there exists a
minimal set of points
$R (m,\varepsilon )= \left(x_{1}^{(m)},\dotsc ,x_{r(m,\varepsilon )}^{{m}} \right)$
such that
\[
X= \cup_{j=1}^{r (m,\varepsilon )} B_{d_{m}} (x_{j}^{(m)},\varepsilon ),
\]
where $d_{m}$ is the distance defined in \eqref{mTdistance}.
Moreover $f$ being continuous, for any $\delta>0 $
there exists a $\varepsilon_{0} (\delta )>0$ such that
\[
d (x,y)<\varepsilon_{0} (\delta ) \ \Rightarrow \ \left| f (x)-f (y)  \right| < \delta .
\]
Now let $\varepsilon <\varepsilon_{0} (\delta )$. Then for any $n,x$ there is a
$0\leq j_{n,x}\leq r (m,\varepsilon )$ such that $T^{n}x\in  B_{d_{m}} (x_{j_{n,x}}^{(m)},\varepsilon )$
so
\[
d \left( T^{n+j}x, T^{j} x_{j_{n,x}}^{(m)} \right) \leq \varepsilon \
\Rightarrow \left| f (T^{n+j}x) -   f (T^{j}x^{(m)}_{j_{n,x}}) \right|
\leq \delta \quad \forall j=0,\dotsc ,m-1.
\]
The second statement is a direct result of the definition of zero topological entropy.
see \eqref{top-entr} and \eqref{zerote}.
\end{proof}

Before going to the proof of the theorem we still need to introduce
a few definitions.

\begin{defn} We will call $X_{n} (x,.)$, $Y_{n}^{m} (x,.)$ the initial random variable
and  new one associated with a $m$-sequence
\begin{align}\label{XYdef}
X_{n} (x,w)&= f (T^{n} x) \mu_{rand} (n) (w)\\
Y_{n}^{(m)} (x,w) &= \frac{1}{m} \sum_{j=0}^{m-1} f ( T^{j}x) \mu_{rand} (n+j) (w)
\end{align}
\end{defn}
Note that  $Y_{n}^{m}$ is a sum of
bounded independent random variables with symmetric distribution
then  lemma \ref{conc-ineq} applies. Moreover
they depend on the parameter $n$ only through the random Moebius function.
These facts will be crucial to control our bounds.

\subsection{\bf Proof of Main result 4 (Theorem (\ref{m})}
We are now able to give the proof of our main result.
\begin{proof}
Using the definitions above the sum we need to estimate can be written as
\begin{equation}\label{eq1}
\frac{1}{N} \sum_{n=0}^{N-1} f (T^{n}x)\mu_{rand} (n) (w) = \frac{1}{N} \sum_{n=0}^{N-1} X_{n} (x,w).
\end{equation}
Since we are interested only in the limit $N\to\infty$ we can replace $X_{n} (x,.)$ by
$Y_{n}^{m} (T^{n}x,.)$. Indeed
\begin{equation}\label{eq3}
\frac{1}{N} \sum_{n=0}^{N-1} X_{n} (x,w)\ = \
\frac{1}{N} \sum_{n=0}^{N-1} Y_{n}^{m} (T^{n}x,w) + R_{Nm} (x,w)
\end{equation}
where
\begin{align}
R_{Nm} (x,w) =& \frac{1}{Nm}\sum_{j=0}^{m-1}
\sum_{n=0}^{j-1}  \left[ f(T^{n}x) \mu_{rand} (n) (w) -  f(T^{n+N}x) \mu_{rand} (n+N) (w)\right] \cr
&
\tend{N}{\infty} 0\label{diff}
\end{align}
for any fixed $m$, since
\[
|R_{Nm}| \leq
\frac{1}{Nm}\sum_{j=0}^{m-1}  j \|f \|_{\infty} \leq  \frac{m}{N} \|f \|_{\infty}
\tend{N}{\infty} 0.
\]
Therefore the difference between the two sequences in
\eqref{eq3} has a well defined limit.
We will consider the second sequence $\frac{1}{N}\sum_{n=0}^{N-1} Y_{n}^{m}$.
let us consider the  $\limsup_N$ and   $\liminf_N$ of this sequence.
We have
\begin{align*}
&\left|  \limsup_{N}  \frac{1}{N} \sum_{n=0}^{N-1} Y_{n}^{m} (x,w) \right|\leq
\limsup_{N} \left|  \frac{1}{N} \sum_{n=0}^{N-1} Y_{n}^{m} (x,w) \right|\\
& \left|  \liminf_{N} \frac{1}{N} \sum_{n=0}^{N-1} Y_{n}^{m} (x,w) \right|\leq
\limsup_{N} \left|  \frac{1}{N} \sum_{n=0}^{N-1} Y_{n}^{m} (x,w) \right|.
\end{align*}
We claim the right-hand side  of these equations is zero. The rest of the proof is
devoted to prove this claim.
Therefore we need to study
\begin{equation}\label{eq2}
 \limsup_{N}\frac{1}{N} \sum_{n=0}^{N-1} \left| Y_{n}^{m} (T^{n}x,w)
 \right|.
\end{equation}
In Lemma \ref{Yb} below we will prove that for each $n\in \mathbb{N}$, and
each $f\in C(X)$ there exists a subset
$\Omega'_{f,n}\subset \Omega$ such that $\mathbb{P} (\Omega'_{f,n})=1$ and
\begin{equation}\label{Ybound}
\limsup_m \sup_{x\in X}  \left| Y_{n}^{m} (T^{n}x,w)\right| =0 \quad
\forall\  w\in \Omega'_{f,n}.
\end{equation}
Now we define $\Omega'_{f}= \cap_{n\in \mathbb{N}} \Omega'_{f,n}$.
This is still a measurable set, $\mathbb{P} (\Omega'_{f})=1$ and for
each $ w\in \Omega'_{f}$ we have
\[
\limsup_m \sup_{x\in X}  \left| Y_{n}^{m} (T^{n}x,w)\right| =0 \quad
\forall \ n\in \mathbb{N}.
\]
Then
\[
\lim_m   Y_{n}^{m} (T^{n}x,w) =0 \qquad
\forall\  n\in \mathbb{N},\ \forall\  x\in X,
\]
and for any $w\in \Omega'_{f}$ we have
\[
 \lim_{m} \frac{1}{N} \sum_{n=0}^{N-1}
 Y_{n}^{m} (T^{n}x,w) \ = 0\qquad  \forall\  N\geq 1,\  \forall\  x\in X.
\]
Exchanging the roles of $m$ and $N$ above we have
\[
\lim_{N} \frac{1}{m} \sum_{j=0}^{m-1}
 Y_{j}^{N} (T^{j}x,w) \ = 0\qquad  \forall m\geq 1,\   \forall\  x\in X.
\]
Note that for any finite $N$ and $m$ we have
\[
 \frac{1}{m} \sum_{j=0}^{m-1}
 Y_{j}^{N} (T^{j}x,w) =  \frac{1}{m} \sum_{j=0}^{m-1}
\frac{1}{N} \sum_{n=0}^{N-1} f (T^{n+j})\mu_{rand} (n+j) (w)
= \frac{1}{N} \sum_{n=0}^{N-1}
 Y_{m}^{n} (T^{n}x,w),
\]
then if  $ w\in \Omega'_{f}$ we have
\[
\lim_{N} \frac{1}{N} \sum_{n=0}^{N-1}
 Y_{m}^{n} (T^{n}x,w) \ = 0.\qquad \forall\  m\geq 1,\ \forall x\in X.
\]
Inserting this result in  \eqref{eq3} we have
 \[
\lim_{N} \frac{1}{N} \sum_{n=0}^{N-1}
 X_{n} (x,w) \ = 0.\qquad  \forall\  x\in X.
\]
for each $ w\in \Omega'_{f}$.
To complete the proof we need to get rid of the $f$ dependence
in the set $ \Omega'_{f}$.
But, since $X$ is a compact set, the space $C(X)$ of continuous functions on
$X$ is separable. Then, there exists a sequence of continuous functions
$(f_r)_{r \geq 0}$ dense in $C(X)$. We define
\[
\Omega'=\bigcap_{r \in \N}\Omega'_{f_r}.
\]
This set satisfies $\mathbb{P} (\Omega')=1$. Now let $w\in \Omega'$ and
$f\in C(X)$. For any $\delta>0$ we can find a function $f_{r}$ in the sequence
such that $\|f-f_r\|_\infty<\delta $. Then
\[
\lim_{N}\left | \frac{1}{N} \sum_{n=0}^{N-1} f (T^{n}x) \mu_{rand} (n) (w) \right |
\leq \delta + \lim_{N}\left | \frac{1}{N} \sum_{n=0}^{N-1} f_{r} (T^{n}x) \mu_{rand} (n) (w) \right |
= \delta.
\]
Letting $\delta \to 0$ we have for each $w\in \Omega'$
\[
\frac{1}{N}\sum_{ n=0}^{N-1} f(T^nx)\mu_{rand}(n) (w) \tend{N}{\infty}0 \qquad
\forall\  x\in X,\  \forall \  f\in C (X).
\]
This completes the proof of  the theorem.
\end{proof}
Finally we give the proof of the bound  \eqref{Ybound}, that we
used in the proof above.
\begin{lemm}\label{Yb}
Let
\[
Y_{n}^{(m)} (x,w) = \frac{1}{m} \sum_{j=0}^{m-1} f ( T^{j}x) \mu_{rand} (n+j) (w).
\]
Then for each $n\in \mathbb{N}$, and
each $f\in C(X)$ there exists a subset
$\Omega'_{f,n}\subset \Omega$ such that $\mathbb{P} (\Omega'_{f,n})=1$ and
\begin{equation}
\limsup_m \sup_{x\in X}  \left| Y_{n}^{m} (T^{n}x,w)\right| =0 \qquad
\forall\  w\in \Omega'_{f,n}.
\end{equation}
\end{lemm}
\begin{proof}
Let $\delta,\eta >0$ be two fixed positive constants. We will adjust later
$\eta $ as a function of $\delta $. Let $R (m,\varepsilon)$ as in \eqref{Rset}
and $\varepsilon<\min[\varepsilon_{0} (\delta ),\varepsilon_{1} (\eta )  ] $
as in lemma \ref{entropy} above. Then for any $x,n$ there exists a
$0\leq j_{n,x}\leq r (m,\varepsilon )$ such that
\[
 \left| f (T^{n+j}x) -   f (T^{j}x^{(m)}_{j_{n,x}}) \right|
\leq \delta \quad \forall j=0,\dotsc ,m-1.
\]
Inserting this bound in $Y^{m}_{n}$ we have
\begin{align*}
\left| Y_{n}^{m} (T^{n}x,w) \right| &\leq
\left| Y_{n}^{m} (T^{n}x,w)- Y_{n}^{m} (x^{(m)}_{j_{n,x}},w)  \right|
+ \left| Y_{n}^{m} (x^{(m)}_{j_{n,x}},w)  \right|\cr
& <\  \delta +
\sup_{1\leq k\leq r (m,\varepsilon )} \left| Y_{n}^{m} (x^{(m)}_{k},w)  \right|.
\end{align*}
The right-hand side does not depend on $x$ so we also have
\[
\sup_{x\in X} \left| Y_{n}^{m} (T^{n}x,w)  \right| < \delta +
\sup_{1\leq k\leq r (m,\varepsilon )} \left| Y_{n}^{m} (x^{(m)}_{k},w)  \right|.
\]
Now remark that $Y^{m}_{n} (x^{(m)}_{k},.)$ is a sum of independent bounded random
variables with zero average and symmetric distribution so lemma   \ref{conc-ineq} applies.
\begin{equation}
\mathbb{P} \left( \left|Y^{m}_{n} (x)  \right|>t \right) =
\mathbb{P} \left( \left|\sum_{j=0}^{m-1} f (T^{j}x)\mu_{rand} (n+j)\right|>tm \right)
\leq \ 2 e^{-\frac{mt^{2}}{2\|f\|^{2}_\infty}}
\end{equation}
where we used $\sum_{j=0}^{m-1} c_{j}^{2}\leq m\|f\|^{2}_\infty$.
Then
\begin{align}
\mathbb{P} \left( \sup_{1\leq k\leq r (m,\varepsilon )} \left|Y^{m}_{n} (x_{k})  \right|>t \right)
&\leq   \sum_{k=1}^{ r (m,\varepsilon )} \mathbb{P} \left( \left|Y^{m}_{n} (x_{k})  \right|>t \right)
\leq \ 2\, e^{\eta m}   e^{-\frac{mt^{2}}{2\|f\|^{2}_\infty}}\cr
& = \ 2\, e^{- m \left[  \frac{t^{2}}{2\|f\|^{2}_\infty}- \eta  \right]}
\end{align}
In the following we fix $t =2\delta $ and $\eta= \frac{\delta^{2} }{\|f\|^{2}_\infty}$. Then
\[
\sup_{x\in X} \left| Y_{n}^{m} (T^{n}x)  \right| \leq  \delta +t
= 3 \delta \quad \mbox{with probability larger than}
\ 1- 2e^{-m \frac{\delta^{2} }{2\|f\|^{2}_\infty } }.
\]
Finally to get rid of the $m$ dependence, we consider
\begin{align*}
& \mathbb{P} \left( \sup_{m\geq M} \sup_{x\in X}\left|Y^{m}_{n} (T^{n}x)  \right|> 3\delta\right)   \\
& \leq
\sum_{m\geq M} \mathbb{P} \left(\sup_{x\in X}\left|Y^{m}_{n} (T^{n}x)  \right|> 3\delta  \right)
\leq   \sum_{m\geq M} 2\, e^{-m \frac{\delta^{2} }{2\|f\|^{2}_\infty } }< \delta. \nonumber
\end{align*}
The last inequality is true only for $M$ large enough.
But, we can always choose $M (\delta )$ such that
the sum above is bounded by $\delta $.
Inserting all these results in the initial sum we find that $\forall M\geq M (\delta )$
\begin{equation}
 \sup_{m\geq M} \sup_{x\in X} \left| Y_{n}^{m} (T^{n}x)
 \right| \leq 3\delta
\end{equation}
with probability larger than $1-\delta $.
Therefore
\begin{equation}
 \limsup_{m} \sup_{x\in X} \left| Y_{n}^{m} (T^{n}x)
 \right| \leq 3\delta, \quad \forall M\geq M (\delta )
\end{equation}
with probability larger than $1-\delta $.
Let us take the sequence $\delta_{q}=\frac{1}{q^{2}}$ and let
\[
E_{q,n} = \{  \limsup_{m} \sup_{x}
\left|Y^{m}_{n} (T^{n}x)  \right|> 3\delta_{q}  \}.
\]
Then we have
\[
\sum_{q} \mathbb{P} \left(E_{q,n} \right)\leq  \sum_{q} \delta_{q} < \infty
\]
and by Borel-Cantelli lemma \ref{Borel-Cantelli Lemma}
\[
 \mathbb{P}  \left(\limsup_q E_{q,n} \right)=0.
\]
 This means that
\[
1=\mathbb{P}  \left(\liminf_q E^{c}_{q,n} \right)= \mathbb{P}
\left(\cup_{q_{0}} \cap_{q\geq q_{0}} E^{c}_{q,n} \right).
\]
We define
\[
\Omega'_{n,f}=  \cup_{q_{0}} \cap_{q\geq q_{0}} E^{c}_{q,n}.
\]
Then $\mathbb{P}(\Omega'_{f})=1$ and for each $w\in \Omega'_{f}$
there exists a $q_{0}$ such that
\[
 \limsup_{m} \sup_{x\in X} \left| Y_{n}^{m} (T^{n}x,w)
 \right| \leq  3\delta_{q}\qquad  \forall q\geq q_{0}.
\]
Then for all $w\in \Omega'_{n,f}$
\[
 \limsup_{m} \sup_{x\in X}\left| Y_{n}^{m} (T^{n}x,w)
 \right|=0.
 \]
This concludes the proof of the lemma.
\end{proof}

\appendix

\section{Concentration inequality}\label{appA}

We include here a proof of a useful concentration
inequality due to Hoeffding and Azuma \cite{Azuma},\cite{Mcdiarmid}.

\begin{thm}[Hoeffing-Azuma Inequality]\label{HA}
Let $\{Y_k,\F_k\}_{k=1}^{n}$ be a martingale difference
sequence (i.e., $Y_k$ is $\F_k$-measurable, $E[|Y_k|] < +\infty$ and $\E[Y_k|\F_{k-1}] = 0$ a.s.
for every $k \leq n)$. Assume that, for every $k \in \{1,\cdots,n\}$, there exist numbers
$c_k \in \R_{+}$ such that a.s. $|Y_k| \leq c_k$. Then, for every $t > 0,$
\[
\P\Big\{\big|\sum_{k=1}^{n}Y_k\big|\geq t\Big\} \leq 2 e^{-\frac{t^{2}}{2\sum_{k=1}^{n}c_k^2 }  }
\]
\end{thm}

To prove this theorem we will need the following lemma.
\begin{lemm}\label{lA1}
Let $f$ be a convex smooth real function such that, for any
$n \in \N$, $f^{(2n)}(0)=f^{(n)}(0) > 0$, $\F$  a sub $\sigma$-algebra and  $Y$  a
random variable such that almost surely $ \E[Y|\F] = 0$ and $|Y| \leq c$ for some constant $c > 0$.
Then for any $t > 0$ we have
\[
\E(f(tY)|\F) \leq f\left(\frac{t^{2}c^{2}}{2} \right)~~~{\textrm{a.s.}}.
\]
\end{lemm}
\begin{proof} Since $f$ is convex we have
\[
f\left( \alpha x_{1}+ (1-\alpha )x_{2}  \right)\leq  \alpha f (x_{1})+ (1-\alpha )f (x_{2})
\quad \forall x_{1},x_{2}\in \mathbb{R},  \forall\, 0\leq  \alpha \leq 1.
\]
Then, if $|y| \leq c$, we can choose $x_{1}=c$ and $x_{2}=-c$, so
\[
f(y) \leq \frac{1}{2}\Big\{1+\frac{y}{c}\Big\}f(c)+ \frac{1}{2}\Big\{1-\frac{y}{c}\Big\}f(-c).
\]
Taking a conditional expectation with respect to $\F$ on both sides, we get
\[
\E(f(Y)|\F) \leq \frac{1}{2}\big\{f(c)+f(-c)\big\}.
\]
Now expanding in a Taylor series, we have
\[
f(-c)+f(c)=\sum_{n=0}^{+\infty}\frac{f^{(2n)}(0)}{2n!}c^{2n}.
\]
But, for every integer $m \geq 0$,
$$2n! \geq 2n(2n-2)(2n-4)\cdots 2=2^n n!$$
then due to the our assumption $f^{(2n)}(0)=f^{(n)}(0) > 0$ , we have
\[
f(-c)+f(c)\leq \sum_{n=0}^{+\infty}\frac{f^{(n)}(0)}{n!}\Big(\frac{c^2}2\Big)^n=f(c^2/2),
\]
for any choice of $c\in \mathbb{R}$. Finally
\[
\E(f(Y)|\F) \leq f\left(\frac{c^{2}}{2} \right).
\]
Replacing $Y$ by $tY$ and $c$ by $tc$ we obtain the result.
\end{proof}

\begin{proof}[\bf Proof of Theorem \ref{HA}]
By Markov inequality we have
\begin{equation}\label{Azuma}
\P\left \{\sum_{k=1}^{n}Y_k \geq t\right \}= \P\left\{\sum_{k=1}^{n}Y_k- t \geq 0\right\} \leq
e^{-\lambda t} \E\left(e^{\lambda\sum_{k=1}^{n}Y_k}\right).
\end{equation}
for every $\lambda>0$. Now
\begin{eqnarray*}
\E\Big(\exp\Big(\lambda\sum_{k=1}^{n}Y_k\Big)\Big) &=&
\E\Big(\E\Big(\exp\Big(\lambda\sum_{k=1}^{n}Y_k\Big)|\F_{n-1}\Big)\Big)\\
&=&\E\Big(\exp\Big(\lambda\sum_{k=1}^{n-1}Y_k\Big)\E\Big(\exp(\lambda Y_n)|\F_{n-1}\Big)\Big)
\end{eqnarray*}
where the last passage holds since $\exp\Big(\lambda\sum_{k=1}^{n-1}Y_k\Big)$ is $\F_{n-1}$-measurable.
Applying Lemma \ref{lA1}, we get
\begin{eqnarray}\label{rec}
\E\Big(\exp\Big(\lambda\sum_{k=1}^{n}Y_k\Big)\Big) &=&
\E\Big(\exp\Big(\lambda\sum_{k=1}^{n-1}Y_k\Big)\E\Big(\exp(\lambda Y_n)|\F_{n-1}\Big)\Big) \nonumber\\
&\leq & \E\Big(\exp\Big(\lambda\sum_{k=1}^{n-1}Y_k\Big)\Big) \exp(\lambda^2 c_n^2/2)
\end{eqnarray}
Hence, by repeatedly using the recursion in \eqref{rec}, it follows that
\[
\E\Big(\exp\Big(\lambda\sum_{k=1}^{n}Y_k\Big)\Big) \leq \prod_{k=1}^{n}\exp\big(\lambda^{2} c_k^2/2\big)=
\exp\Big(\lambda^{2}\sum_{k=1}^{n}{c_k^2}/2\Big).
\]
Inserting this result in \eqref{Azuma} we have
\[
\P\Big\{\sum_{k=1}^{n}Y_k \geq t\Big\} \leq  \exp\Big(-\lambda t+\lambda^{2}\sum_{k=1}^{n}{c_k^2}/2\Big)
\]
for every $\lambda >0$. Optimizing over the free parameter $\lambda>0$ we obtain
\[
\P\Big\{\sum_{k=1}^{n}Y_k \geq t\Big\} \leq \exp\Big(-t^2/(2\sum_{k=1}^{n}c_k^2)\Big).
\]
Therefore, by substituting $-Y$ by $Y$ and using the same procedure, we get
 \[
\P\Big\{\sum_{k=1}^{n}Y_k \leq -t\Big\} \leq \exp\Big(-t^2/(2\sum_{k=1}^{n}c_k^2)\Big).
\]
Hence
\[
\P\Big\{\Big|\sum_{k=1}^{n}Y_k\Big| \geq t\Big\} \leq 2\exp\Big(-t^2/(2\sum_{k=1}^{n}c_k^2)\Big),
\]
This proves the theorem.
\end{proof}

\section*{Acknowledgements} The first author would like to thank M.
 Lema\'nczyk and J.P. Thouvenot for many useful discussions on the subject.



\begin{thebibliography}{99}



\bibitem{elabdal-lem-de-la-rue}
E. H. e. Abdalaoui, M. Lema\'nczyk and T. de la Rue, \
\emph{On spectral disjointness of powers for rank-one transformations and M\"obius orthogonality},

 http://arxiv.org/abs/1301.0134.
%
\bibitem{elabdal-lem}
E. H. e. Abdalaoui, M. Lema\'nczyk,
\emph{The M\"obius function is orthogonal to any generalized Morse sequence}, preprint.
%
\bibitem{AKL}
E. H. e. Abdalaoui, S. Kasjan and M. Lema\'nczyk,
\emph{0-1 sequences of the Thue-Morse type and Sarnak's conjecture}, preprint.
http://arxiv.org/abs/1304.3587.
\bibitem{Azuma}
K. Azuma, \emph{Weighted sums of certain dependent random variables,}
T{\^o}hoku Math. J. (2) 19 1967 357-367.
%
\bibitem{Bellow-Losert}
A. Bellow, V. Losert, \emph{ The weighted pointwise ergodic theorem and the individual ergodic theorem along
subsequences,} Trans. Amer. Math. Soc. 288 (1985), no. 1, 307--345.
%
\bibitem{Bhatta}
A. Bhattacharyya, \emph{On a measure of divergence between two statistical populations defined by their
probability distributions}, Bulletin of the Calcutta Mathematical Society 35 (1943), 99--109.
%
\bibitem{Bo}J.\ Bourgain, {\em On the correlation of the M\"obius function with random rank-one systems},
J. d'Analyse~Math (to appear). arXiv:1112.1032.
%
 \bibitem{Bo2} J.\ Bourgain, {\em On the spectral type of Ornstein class one transformations},
Israel J.\ Math.\ {\bf 84} (1993), 53--63.
%
\bibitem{Bourgain-sarnak-Ziegler}
J.\ Bourgain, P.\ Sarnak and T.\ Ziegler,
{\em Disjointness of M\"obius from horocycle flows},
in From Fourier and Number Theory to Radon Transforms and Geometry, in memory of Leon Ehrenpreiss,
{\em Developments in Mathematics}, {\bf 28} (2012), 67--83, Springer Verlag. arXiv:1110.0992.
%
\bibitem{Bowen}
R. Bowen, \emph{Entropy for group endomorphisms and homogeneous spaces,}
Trans. Amer. Math. Soc. 153 1971 401--414.
%
\bibitem{Cellarosi-Sinai}
F.~Cellarosi, Ya.G.~Sinai, \emph{Ergodic Properties of Square-Free Numbers,}

arXiv:1112.4691.
%
\bibitem{Coquet}
J. Coquet, \emph{ Corr\'elation de suites arithm\'etiques,} (French) S\'eminaire Delange-Pisot-Poitou,
20e année: 1978/1979. Th\'eorie des nombres, Fasc. 1 (French), Exp. No. 15, 12 pp., Secr\'etariat Math.,
Paris, 1980.
%
\bibitem{Coquet-France}
 J. Coquet, T. Kamae, M. Mend\`es-France, \emph{Sur la mesure spectrale de certaines suites arithm\'etiques},
Bull. Soc. Math. France, 105 (1977), 369--384.
%
\bibitem{DaCa} D. Dacunha-Castelle, M. Duflo, \emph{Probabilit\'es et statistiques, 2.},
Masson,  (1983).
%
\bibitem{Da} H. Davenport, \emph{On some infinite series involving arithmetical functions. II},
Quart. J. Math. Oxf. \textbf{8} (1937),  313--320.
%
\bibitem{Dunford-Schawrtz}
N. Dunford and T. Schwartz, \emph{ Linear operators. Part II. Spectral theory. Selfadjoint operators
in Hilbert space. }
With the assistance of William G. Bade and Robert G. Bartle. Reprint of the 1963 original.
Wiley Classics Library. A Wiley-Interscience Publication. John Wiley \& Sons, Inc., New York, 1988. pp. i--x,
859--1923 and 1--7.
%
\bibitem{Elliott-C}
 P. D. T. A. Elliott, \emph{On the correlation of multiplicative and the sum of additive arithmetic functions,} M
em. Amer. Math. Soc. 112 (1994), no. 538, viii+88 pp.
%
\bibitem{Green-Tao}
B. Green, T. Tao,  \emph{The M{\"o}bius function is strongly orthogonal to nilsequence,}.
Ann. of Math. (2) 175 (2012), no. 2, 541--566.
%
\bibitem{Hildebrand}
A.J. Hildebrand, \emph{Introduction to analytic number theory,} Math 531 lectures notes Fall 2005,
http://www.math.uiuc.edu/$\sim$hildebr/ant/.
%
\bibitem{Jewett}
R.I. Jewett, \emph{The prevalence of uniquely ergodic systems,} J. Math. Mech. 19 (1970) 717-729.
%
\bibitem{Kamae}
T. Kamae, \emph{Sum of digits to different bases and mutual singularity of their spectral mesures},
Osaka J. Math., 15 (1978), 569--574.
%
\bibitem{Matusita1}
 K. Matusita, \emph{Decision rules, based on the distance for problems of fit, two samples, and estimation},
Ann. Math. Statist. 26 (1955), 631--640.
%
 \bibitem{Matusita2}
K. Matusita, \emph{A distance and related statistics in multivariate analysis,} Multivariate Analysis
(Proc. Internat. Sympos., Dayton, Ohio, 1965), 1966, pp. 187--200.
%
 \bibitem{Matusita3}
K. Matusita, \emph{ On the notion of affinity of several distributions and some of its applications,}
Ann. Inst. Statist. Math. 19 1967 181--192.
%
\bibitem{Mcdiarmid}
 C. McDiarmid, \emph{Concentration. Probabilistic methods for algorithmic discrete mathematics,} 195--248,
Algorithms Combin., 16, Springer, Berlin, 1998.
%
\bibitem{Mirsky}
 L. Mirsky, \emph{ On the frequency of pairs of square-free numbers with a given difference,}
Bull. Amer. Math. Soc. 55, (1949). 936--939.
%
\bibitem{Nadkarni}
M. G. Nadkarni, \emph{Spectral theory of dynamical systems,} Reprint of the 1998 original.
Texts and Readings in Mathematics, 15. Hindustan Book Agency, New Delhi, 2011.
%
\bibitem{Ng}
N. Ng, \emph{The M\"obius function in short intervals,} Anatomy of integers, 247--257,
CRM Proc. Lecture Notes, 46, Amer. Math. Soc., Providence, RI, 2008.
%
\bibitem{parry}
W. Parry, \emph{Topics in ergodic theory,} Reprint of the 1981 original. Cambridge Tracts in Mathematics,
75. Cambridge University Press, Cambridge, 2004.
%
\bibitem{PYuri}
M. Pollicott, M. Yuri, \emph{Dynamical systems and ergodic theory,} London Mathematical Society Student Texts,
40. Cambridge University Press, Cambridge, 1998.
%
\bibitem{Queffelec1}
M. Queffelec, \emph{Mesures spectrales associ\'ees \`a certaines suites arithm\'etiques},
Bull. Soc. Math. France, 107 (1979), 385-421.
%
\bibitem{Queffelec2}
M.~Queffelec, \emph{Sur la singularit{\'e} des produits de Riesz et des mesures spectrales associ{\'e}es {\`a}
la somme des chiffres,} (French) Israel J. Math. 34 (1979), no. 4, 337--342 (1980).
%
\bibitem{sarnak1}
P. Sarnak,  \emph{M{\"o}bius Randomness and Dynamics}, Lecture Slides Summer 2010.

http://publications.ias.edu/sarnak/paper/518.
%
\bibitem{sarnak2}
P. Sarnak,  \emph{M{\"o}bius Randomness and Dynamics}, Mahler Colloquium Lectures 2011,
http://publications.ias.edu/sarnak/paper/546.
%
\bibitem{sarnak}
P. Sarnak,  \emph{Three lectures on the M{\"o}bius function randomness and dynamics
(Lecture 1)}. http://publications.ias.edu/sarnak/paper/506.
%
\bibitem{Titchmarsh}
E. C. Titchmarsh, \emph{The theory of the Riemann zeta-function,} Second edition. Edited and with a preface
by D. R. Heath-Brown. The Clarendon Press, Oxford University Press, New York, 1986.
%
\bibitem{Walters}
P. Walters, \emph{An introduction to ergodic theory. Graduate Texts in Mathematics,} 79. Springer-Verlag,
New York-Berlin, 1982.
%
\bibitem{Wiener}
N. Wiener,\emph{The Fourier integral and certain of its applications,} Dover publications, (1958).
%
\bibitem{Wiener-Wintner}
N. Wiener, A. Wintner,\emph{Harmonic analysis and ergodic theory,} American Journal of Mathematics 63 (1941):
415--426.
%
\bibitem{Ziegler} T. Zeigler, \emph{A soft proof of orthogonality of M{\"o}bius to nilfows},

http://www.technion.ac.il/$\sim$tamarzr/soft-green-tao.pdf.
\end{thebibliography}
\end{document}